\documentclass[12 pt]{article}%
\usepackage{amsmath, amsfonts, amsthm, color,latexsym}
\usepackage{amsmath}
\usepackage{amsfonts}
\usepackage{amssymb}
\usepackage{color, soul}
\usepackage{enumerate}
\providecommand{\U}[1]{\protect\rule{.1in}{.1in}}
\newtheorem{theorem}{Theorem}[section]
\newtheorem{proposition}[theorem]{Proposition}
\newtheorem{corollary}[theorem]{Corollary}
\newtheorem{example}[theorem]{Example}

\newtheorem{remark}[theorem]{Remark}
\newtheorem{remarks}[theorem]{Remarks}

\newtheorem{final remark}[theorem]{Final Remark}
\newtheorem{definition}[theorem]{Definition}
\allowdisplaybreaks[4]

\begin{document}
	
	%\title{\sc Weak and Weak* Dunford Pettis Multilinear Operators}
	\title{\sc Dunford–Pettis Multilinear Operators and Their Weak Variant on Banach Spaces}
	\author{Joilson Ribeiro,%\thanks{joilsonor@ufba.br}
		~ Fabr\'icio Santos }
	\date{}
	\maketitle

	\begin{abstract}
		In this paper, we address broader concepts of Dunford-Pettis operators, presenting new classes and results that relate this class to others already well studied in the literature, as well as an approach outside the origin. We also investigate inclusion results and conditions for the coincidence of this class with previously studied classes and with new classes that arise throughout this work.
	\end{abstract}
	
	\section{Introduction and background}

	The class of Dunford–Pettis linear operators has played a crucial role in the theory of Banach spaces in general, as highlighted in \cite{DU77}. Several researchers have investigated this topic, mainly seeking correlations among the classes of compact linear operators, Dunford–Pettis operators, weakly Dunford–Pettis operators, and weakly* Dunford–Pettis operators, as can be seen in \cite{AB82, BHJ22, BKAB24}.
	
	A first multilinear approach to Dunford-Pettis operators between Banach spaces was proposed by Ribeiro, Santos, and Torres in \cite{RST22}, where it was shown that this class is an ideal of multilinear mappings but not a hyper-ideal, a concept introduced by Botelho and Torres in \cite{BT15}. In that work, Ribeiro, Santos, and Torres focused on verifying that this class does not satisfy the coherence concept introduced by Pellegrino and Ribeiro in \cite{PR14}. However, they did not explore other generalizations to the multilinear setting to study inclusion and coincidence results. It is worth noting that the case we will call Dunford-Pettis multilinear operators at every point already appears in the literature and can be found in \cite{AD02, BR98}.
	 
	In this paper, in addition to introducing new classes that generalize the class of Dunford-Pettis linear operators, we also revisit concepts such as multi-ideals, hyper-ideals, and coherence and compatibility to formalize these new structures into properties considered fundamental for their development, such as closure under composition with linear operators and the product of linear functionals and multilinear forms. We also present coincidence conditions for these classes and seek characterizations that make their use smoother.	

	Let $E$ and $F$ be Banach spaces. We indicate by $x_j \xrightarrow{w} x$ to mean that $x_j \to x$ in $\sigma(E,E^{'})$, as $j \to \infty$. We say that $E$ has the Schur property if $\|x_j\| \to 0$ whenever $x_j \xrightarrow{w} 0$ in $E$, and that $E$ has the Dunford--Pettis property if $f_j(x_j) \to 0$ whenever $x_j \xrightarrow{w} 0$ in $E$ and $f_j \xrightarrow{w} 0$ in $E'$. An operator $T\colon E \rightarrow F$ is called a Dunford–Pettis operator, and is denoted by $\mathcal{DP}$, if $T$ carries weakly convergent sequences to norm convergent sequences. In other words, if $x_j \to 0$ as $j \to \infty$ in $\sigma(E,E')$, then $\|T(x_j)\|_F \to 0$ as $j \to \infty$.
	
Alternatively, a (bounded linear) operator $T\colon E \rightarrow F$ is a Dunford–Pettis operator if and only if $T$ carries relatively weakly compact sets to norm totally bounded sets. A Banach space $E$ is said to have the Schur property if $Id_E\colon E \rightarrow E$, given by $Id_E(x)=x$, is a Dunford–Pettis operator. The Banach space $\ell_1$ has the Schur property (\cite[Theorem 2.5.18]{MN91}).
	
	We can also define the class of weakly Dunford–Pettis linear operators as follows: Let $E$ and $F$ be Banach spaces. An operator $T\colon E \rightarrow F$ is said to be weakly Dunford–Pettis if whenever $x_j \to 0$ in $\sigma(E,E')$ and $f_j \to 0$ in $F'$, then $f_j(T(x_j)) \to 0$. We denote this class by $\mathcal{WDP}$ (see, \cite{AB85}).
	
	From now on, the letters $E,E_{1},\dots,E_{m},F,G,H$ will denote Banach spaces over the same scalar field $\mathbb{K}=\mathbb{R}$ or $\mathbb{C}$, and $E'$ stands for the topological dual of $E$. The symbol $B_E$ denotes the closed unit ball of $E$. We use $\mathrm{BAN}$ to denote the class of all Banach spaces over $\mathbb{K}$.
	
	By $\mathcal{L}(E_1, \ldots, E_n;F)$ we denote the Banach space of continuous $n$-linear operators from $E_1 \times \cdots \times E_n$ to $F$, endowed with the usual uniform norm $\|\cdot\|$. In the linear case we write $\mathcal{L}(E;F)$. If $E_1 = \cdots = E_n = E$, we write $\mathcal{L}(^nE;F)$. If $F = \mathbb{K}$, we write $\mathcal{L}(E_1, \ldots, E_n)$ and $\mathcal{L}(^nE)$.
	
	A mapping $P \colon E \rightarrow F$ is said to be an {\it $m$-homogeneous polynomial} if there exists $A \in \mathcal{L}(^mE; F)$ such that $P(x) = A(x)^m$ for every $x \in E$, where $A(x)^m := A\left(x,\overset{m}{\dots},x\right)$. In this case we write $P = \widehat{A}$. By $\mathcal{P}(^mE;F)$ we denote the Banach space of all continuous $m$-homogeneous polynomials from $E$ to $F$, endowed with the usual supremum norm $\|\cdot\|$. By $\check{P}$ we denote the unique symmetric continuous $m$-linear operator associated with $P$.

Let $\mathcal{I}$ be an operator ideal. An $n$-linear mapping $T \in \mathcal{L}(E_1,\dots, E_n; F)$ belongs to $\mathcal{I}\circ \mathcal{L}$ -- in this case
we write $T \in \mathcal{I} \circ \mathcal{L}(E_1,\dots, E_n; F)$ -- if there are a Banach space $G$, an nlinear
mapping $B \in \mathcal{L}(E_1,\dots, E_n; G)$ and an operator $u \in \mathcal{I}(G; F)$ such that
$T = u \circ B$ (for more details see, \cite{BPR07}).

\begin{definition}(Coherent pair of ideals \cite{PR14})
Let $N \in \mathbb{N}\cup {\infty}$. The sequence $\left(\mathcal{P}_n, \mathcal{M}_n \right)_{n=1}^N$, with $\mathcal{P}_1 = \mathcal{M}_1 = \mathcal{M}$ is coherent with $\mathcal{M}$ if there exist positive constants $\beta_1, \beta_2, \beta_3$ such that for all Banach spaces $E, E_1,\dots, E_n, F$, the following conditions hold for all $n \in \{1,\dots, N-1\}$:

\begin{description}
\item $(CH 1)$ If $T \in \mathcal{M}_{n+1}(E_1,\dots, E_n; F)$ and $a_j \in E_j$ for any $j=1,\dots, n+1$. Then $T_{a_j} \in \mathcal{M}_n(E_1,\dots, E_{j-1}, E_{j+1},\dots, E_{n+1})$ and

\begin{equation*}
\|T_{a_j}\|_{\mathcal{M}_n} \le \beta_1 \|T\|_{\mathcal{M}_{n+1}}\|a_j\|.
\end{equation*}

\item $(CH 2)$ If $P \in \mathcal{P}_{n+1}(E; F)$ and $a \in E$. Then $P_a \in \mathcal{P}_n(E; F)$ and

\begin{equation*}
\|P_a\|_{\mathcal{P}_n} \le \beta_2 \sup\left\{\|\check{P}\|_{\mathcal{M}_{n+1}}, \|P\|_{\mathcal{P}_{n+1}} \right\} \|a\|.
\end{equation*}

\item $(CH 3)$ If $T \in \mathcal{M}_n(E_1,\dots, E_{n}; F)$ and $\varphi \in E_{n+1}'$. Then $\varphi T \in \mathcal{M}_{n+1}(E_1,\dots, E_{n+1}; F)$ and

\begin{equation*}
\|\varphi T\|_{\mathcal{M}_{n+1}} \le \beta_3 \|\varphi \| \|T\|_{\mathcal{M}_n}.
\end{equation*}

\item $(CH 4)$ If $P \in \mathcal{P}_n(E; F)$ and $\varphi \in E'$. Then $\varphi P \in \mathcal{P}_{n+1}(E; F)$.

\item $(CH 5)$ $P$ belong to $\mathcal{P}_n(E; F)$ if, and only if, $\check{P}$ belong to $\mathcal{M}_n(E^n; F)$.
\end{description}
\end{definition}

\section{Weakly Sequentially Continuous Multilinear Operator}\label{DPMOHPEv}

 Recall that the concepts of Dunford–Pettis multilinear operators and homogeneous polynomials (at the origin) were introduced by Ribeiro, Santos, and Torres in \cite{RST22}, under the terminology $\mathcal{L}_{DP}$ and $\mathcal{P}_{DP}$, respectively. A point that will be especially useful is that the class $\mathcal{L}_{DP}$ does not satisfy the coherence condition (CH1) presented in the first section of this manuscript. In that paper it was shown that the coherence condition (CH1) fails, and that this class does not generate a hyper-ideal of multilinear mappings or homogeneous polynomials, concepts introduced in \cite{BT15} and \cite{BT18}. Let us recall this concept.

\begin{definition}
Let $m \in \mathbb{N}$, $E, E_1,\dots, E_m$ and $F$ be Banach spaces. 
\begin{enumerate}
\item[$(a)$] We say that $T \in \mathcal{L}(E_1,\dots, E_m; F)$ is a Dunford-Pettis $m$-linear operator if, for every sequence $(x_j^{(i)})_{j=1}^\infty \subset E_i$, $i=1,\dots, m$, satisfying
	\[
	x_j^{(i)} \overset{w}{\longrightarrow} 0
	\]
	we have
	\[
	T(x_j^{(1)},\dots,x_j^{(m)}) \longrightarrow 0
	\]
	in norm, as $j \to \infty$. In this case, we write $T \in \mathcal{L_{DP}}(E_1,\dots, E_m; F)$.
\item[$(b)$] We say that $P \in \mathcal{P}(E; F)$ is a Dunford-Pettis $m$-homogeneous polynomials if, for every sequence $(x_j)_{j=1}^\infty \subset E$ satisfying
	\[
	x_j \overset{w}{\longrightarrow} 0
	\]
	we have
	\[
	P(x_j) \longrightarrow 0
	\]
	in norm, as $j \to \infty$. In this case, we write $P \in \mathcal{P_{DP}}(^mE; F)$.
\end{enumerate} 
\end{definition}
It is immediate that $\mathcal{L_{DP}}(E_1,\dots, E_m; F) \subset \mathcal{L}(E_1,\dots, E_m; F)$. This inclusion is strict, as we can see in the example below.
\begin{example}
Consider $T \colon \ell_2 \times \ell_2 \to \mathbb{K}$ given by $T(x, y) = \sum_{n=1}^{\infty} x_n y_n$. It is immediate that $T$ is well-defined and $T \in \mathcal{L}(\ell_2, \ell_2; \mathbb{K})$. whenever, if $x_j = y_j = e_j$ we have that $(x_j)_{j=1}^{\infty}$ and $(y_j)_{j=1}^{\infty}$ are weak-null sequences and
\begin{equation*}
T(e_j, e_j) = 1 \not\rightarrow 0.
\end{equation*}
So, $T \notin \mathcal{L_{DP}}(\ell_2, \ell_2, \mathbb{K})$.
\end{example}

\begin{proposition}
Let $E_1,\dots, E_m$ and $F$ be Banach spaces. If $E_i$ has the Schur property, for any $i \in \{1,\dots, m\}$, then
\begin{equation*}
\mathcal{L_{DP}}(E_1,\dots, E_m; F) = \mathcal{L}(E_1,\dots, E_m; F).
\end{equation*}
\end{proposition}

The class of weakly sequentially continuous multilinear operators has been studied by various researchers, among whom we can highlight \cite{AD02, BR98, CGG99}. Let us now revisit this concept. This class is a natural extension of the notion of completely continuous linear operators and allows us to establish relationships between different classes of multilinear mappings associated with continuity and compactness properties.

\begin{definition}\label{DPML}
	Let $m \in \mathbb{N}$, $E_1,\dots, E_m$ and $F$ be Banach spaces. We say that $T \in \mathcal{L}(E_1,\dots, E_m; F)$ is a weakly sequentially continuous $m$-linear operator if, for any $$(a_1,\dots, a_m) \in E_1 \times \cdots \times E_m,$$ and every sequence $(x_j^{(i)})_{j=1}^\infty \subset E_i$, with $i=1,\dots,m$,  satisfying
	\[
	x_j^{(i)} \overset{w}{\longrightarrow} a_i
	\]
	we have
	\[
	T(x_j^{(1)},\dots,x_j^{(m)}) \longrightarrow T(a_1,\dots,a_m)
	\]
	in norm, as $j \to \infty$.
\end{definition}

	The collection of weakly sequentially continuous multilinear operators from $E_1,\dots, E_m$ into $F$ is denoted by $\mathcal{L}_{wsc}(E_1,\dots, E_m; F)$. In the linear case, the term "completely continuous linear operator" is already well established in the literature, and the symbol used to denote these operators is $\mathcal{CC}$.  Analogously, we have the following definition.

\begin{definition}\label{DPPH}
	Let $m \in \mathbb{N}$, let $E$ and $F$ be Banach spaces, and let $P \in \mathcal{P}(^mE; F)$. We say that $P$ is a weakly sequentially continuous $m$-homogeneous polynomial if, for any $a \in E$, and every sequence $(x_j)_{j=1}^{\infty} \subset E$ with
	\begin{equation*}
		x_j \overset{w}{\longrightarrow} a,
	\end{equation*}
	we have
	\begin{equation*}
		P(x_j) \longrightarrow P(a)
	\end{equation*}
	in norm, as $j \to \infty$.
\end{definition}

The collection of weakly sequentially continuous $m$-homogeneous polynomials from $E$ into $F$ is denoted by $\mathcal{P}_{wsc}(^mE; F)$.

\begin{proposition}
	Let $E_1,\dots, E_m$ and $F$ be Banach spaces and $T \in \mathcal{L}(E_1,\dots, E_m; F)$. Then, $T \in \mathcal{L}_{wsc}(E_1,\dots, E_m; F)$ if, and only if, for every $(a_1,\dots, a_m) \in E_1\times\cdots\times E_m$ and every weakly null sequences $(x_j^{(i)})_{j=1}^{\infty}$ in $E_i$, $i=1,\dots, m$, we have
	\begin{equation*}
		\|T(x_j^{(1)} + a_1,\dots, x_j^{(m)} + a_m) - T(a_1,\dots, a_m)\|\longrightarrow 0.
	\end{equation*}
\end{proposition}

\begin{remarks}\label{RIRC}
\begin{enumerate}
\item[$(a)$] It follows immediately from the definitions above that $\mathcal{L}_{wsc} \subset \mathcal{L_{DP}}$ and $\mathcal{P}_{wsc} \subset \mathcal{P_{DP}}$.
\item[$(b)$] It is well established that for any Banach spaces $E_1, \dots, E_m$ and $F$, the collection $\mathcal{L}_{wsc}(E_1,\dots,E_m;F)$ is a closed subset of $\mathcal{L_{DP}}(E_1,\dots,E_m;F)$.
\end{enumerate}
\end{remarks}

Item (a) of Remark \ref{RIRC} allows us to view the class of weakly sequentially continuous multilinear operators as an "everywhere" version of the class of Dunford-Pettis multilinear operators, since it involves essentially the same operational elements. Another point that follows from item $(a)$ of Remarkes \ref{RIRC} is that the inclusion
\begin{equation*}
\mathcal{L}_{wsc} \subset \mathcal{L_{DP}}
\end{equation*}
is strict, as shown by the following example. Let $T\colon \ell_1 \times c_0 \to c_0$ given by $T(x, y) = \varphi(x)y$ where $\varphi \in \ell_1^{'}\setminus\{0\}$. It is immediate that $T \in \mathcal{L}_{DP}(\ell_1, c_0; c_0)$. On the other hand, considering $(x_j)_{j=1}^{\infty} \subset \ell_1$ and $(y_j)_{j=1}^{\infty} \subset c_0$ with $x_j = e_1 \in \ell_1$ e $y_j = e_j \in c_0$ we have
\begin{equation*}
x_j \xrightarrow{w} e_1 \text{ and } y_j \xrightarrow{w} 0.
\end{equation*}
Consider also  $\varphi = e_1^* \in \ell_1^{'}$ where $e_1^*((x_j)_{j=1}^{\infty}) = x_1$, for every $(x_j)_{j=1}^{\infty} \in \ell_1$ then
\begin{equation*}
\|T(e_1, e_j)\|_{\infty} = \|e_1^*(e_1)e_j\|_{\infty} = \|e_j\|_{\infty} = 1. 
\end{equation*}
Thus, $T \notin \mathcal{L}_{wsc}(\ell_1, c_0; c_0)$.

It is straightforward to verify that the pairs $(\mathcal{L}_{wsc}, \|\cdot\|)$ and $(\mathcal{P}_{wsc}, \|\cdot\|)$ form a Banach multi-ideal and a Banach ideal of homogeneous polynomials, respectively (see \cite{FG03, P83} for details). These structures generalize the Banach ideal of linear operators $(\mathcal{CC}, \|\cdot\|)$. On the other hand, $(\mathcal{L}_{wsc}, \|\cdot\|)$ and $(\mathcal{P}_{wsc}, \|\cdot\|)$ are not  hyper-ideal of multilinear mappings and homogeneous polynomials, respectively, as shown in Examples 4.4 and 3.11 of \cite{BT15} and \cite{BT18}, respectively.

We next show that the sequence
\begin{equation*}
	\left(\mathcal{L}_{wsc}^{\mathrm{n}}, \mathcal{P}_{wsc}^{\mathrm{n}} \right)_{n=1}^{\infty}
\end{equation*}
is coherent and compatible with $(\mathcal{L_{DP}}, \|\cdot\|)$, according to \cite{PR14}. However, it is not strongly coherent according to \cite{RST22}. We begin by establishing the following result.

Before proving the next result, we recall a standard notion in multilinear theory. Let $S_m$ denote the group of permutations of $\{1,\dots,m\}$. Given $T \in \mathcal{L}(^mE;F)$ and $\sigma \in S_m$, define $T_{\sigma}$ and $T_S$ in $\mathcal{L}(^mE;F)$ by
\begin{align*}
	T_{\sigma}(x_1,\dots,x_m) &= T(x_{\sigma(1)},\dots,x_{\sigma(m)}),\\
	T_S(x_1,\dots,x_m) &= \frac{1}{m!}\sum_{\sigma \in S_m} T_{\sigma}(x_1,\dots,x_m).
\end{align*}

A subclass $\mathcal{G}$ of the class of continuous multilinear operators between Banach spaces is said to be symmetric if $T_S \in \mathcal{G}(^mE;F)$ whenever $T \in \mathcal{G}(^mE;F)$. If $\mathcal{G}$ is endowed with a function $\|\cdot\|_{\mathcal{G}}\colon \mathcal{G} \to [0,\infty)$, then $\mathcal{G}$ is said to be strongly symmetric if $T_{\sigma} \in \mathcal{G}$ and $\|T_{\sigma}\|_{\mathcal{G}} = \|T\|_{\mathcal{G}}$ for all $T \in \mathcal{G}(^mE;F)$ and $\sigma \in S_m$ (see \cite{FG03}).

\begin{proposition}\label{DPS}
	The subclass $\mathcal{L}_{wsc}$ of the class of continuous multilinear operators between Banach spaces is symmetric.
\end{proposition}
 
%\begin{proof}
%	Let $E_1,\dots,E_m$ and $F$ be Banach spaces, let $\sigma \in S_m$, and let $T \in \mathcal{L}_{wsc}(E_1,\dots,E_m;F)$. Let $(x_j^{(i)})_{j=1}^{\infty} \subset E_i$ be sequences such that
%	\begin{equation*}
%		x_j^{(i)} \overset{w}{\longrightarrow} x_i \in E_i, \quad i=1,\dots,m.
%	\end{equation*}
%	Then
%	\begin{equation*}
%		T(x_j^{(1)},\dots,x_j^{(m)}) \longrightarrow T(x_1,\dots,x_m)
%	\end{equation*}
%	in norm as $j \to \infty$.
	
%	Since $x_j^{(\sigma(i))} \overset{w}{\to} x_{\sigma(i)}$ for $i=1,\dots,m$, it follows that
%	\begin{equation*}
%		\|T_{\sigma}(x_j^{(1)},\dots,x_j^{(m)}) - T_{\sigma}(x_1,\dots,x_m)\|
%		= \|T(x_j^{(\sigma(1))},\dots,x_j^{(\sigma(m))}) - T(x_{\sigma(1)},\dots,x_{\sigma(m)})\| \to 0
%	\end{equation*}
%	as $j \to \infty$. Hence,
%	\begin{equation*}
%		\|T_S(x_j^{(1)},\dots,x_j^{(m)}) - T_S(x_1,\dots,x_m)\| \to 0
%	\end{equation*}
%	as $j \to \infty$. Therefore, $T_S \in \mathcal{L}_{wsc}(E_1,\dots,E_m;F)$.
%\end{proof}

\begin{theorem}\label{DPcoercomp}
 The sequence $\left(\mathcal{L}_{wsc}^n, \mathcal{P}_{wsc}^n \right)_{n=1}^{\infty}$ is coherent and compatible with $\mathcal{CC}$.
\end{theorem}

\begin{proof}
		
Let $m \in \mathbb{N}$, and let $E$ and $F$ be Banach spaces. We begin by showing that, for any $k \in \{1,\dots,m\}$, $P \in \mathcal{P}_{wsc}^k(^kE;F)$ if and only if $\check{P} \in \mathcal{L}_{wsc}^k(^kE;F)$.

We begin by showing the condition $(CH5)$. Suppose that $\check{P} \in \mathcal{L}_{wsc}^k(^kE;F)$. Let $(x_j)_{j=1}^{\infty} \subset E$ be such that $x_j \overset{w}{\to} x \in E$. Then
\begin{equation*}
	P(x_j) = \check{P}(x_j,\dots,x_j) \longrightarrow \check{P}(x,\dots,x) = P(x)
\end{equation*}
in norm, and hence $P \in \mathcal{P}_{wsc}^k(^kE;F)$.

Conversely, suppose that $P \in \mathcal{P}_{wsc}^k(^kE;F)$. Let $(x_j^{(i)})_{j=1}^{\infty} \subset E$ be such that $x_j^{(i)} \overset{w}{\to} x_i \in E$, for all $i=1,\dots,k$. Since $\epsilon_i = \pm 1$, we have
\begin{equation*}
	\epsilon_1 x_j^{(1)} + \cdots + \epsilon_k x_j^{(k)} \overset{w}{\longrightarrow} \epsilon_1 x_1 + \cdots + \epsilon_k x_k.
\end{equation*}
Hence,
\begin{align*}
	&\left\|\check{P}(x_j^{(1)},\dots,x_j^{(k)}) - \check{P}(x_1,\dots,x_k)\right\| \\
	&= \frac{1}{2^k k!} \left\| P\!\left(\sum_{i=1}^k \epsilon_i x_j^{(i)}\right) - P\!\left(\sum_{i=1}^k \epsilon_i x_i\right) \right\| \longrightarrow 0
\end{align*}
as $j \to \infty$. Thus, $\check{P} \in \mathcal{L}_{wsc}^k(^kE;F)$.

Next, we show the condition $(CH1)$, that is, that for any $T \in \mathcal{L}_{wsc}^m(E_1,\dots,E_m;F)$ and $a_i \in E_i$, we have $T_{a_i} \in \mathcal{L}_{wsc}^{m-1}(E_1,\dots,E_{i-1},E_{i+1},\dots,E_m;F)$ and
\begin{equation*}
	\|T_{a_i}\| \le \|T\| \|a_i\|,
\end{equation*}
for all $i=1,\dots,m$. Without loss of generality, consider $i=1$. Let $(x_j^{(i)})_{j=1}^{\infty} \subset E_i$ be such that $x_j^{(i)} \overset{w}{\to} x_i$ for $i=2,\dots,m$. Then
\begin{equation*}
	\|T_{a_1}(x_j^{(2)},\dots,x_j^{(m)}) - T_{a_1}(x_2,\dots,x_m)\|
	= \|T(a_1,x_j^{(2)},\dots,x_j^{(m)}) - T(a_1,x_2,\dots,x_m)\| \to 0,
\end{equation*}
since the constant sequence $(x_j^{(1)})_{j=1}^{\infty} \subset E_1$, with $x_j^{(1)} = a_1$, converges weakly to $a_1$. The norm inequality is immediate.

Finally, we show the condition $(CH3)$. Let $T \in \mathcal{L}_{wsc}^m(E_1,\dots,E_m;F)$ and $\gamma \in E_{m+1}'$. We show that $\gamma T \in \mathcal{L}_{wsc}^{m+1}(E_1,\dots,E_{m+1};F)$. Let $(x_j^{(i)})_{j=1}^{\infty} \subset E_i$ be such that $x_j^{(i)} \overset{w}{\to} x_i \in E_i$ for all $i=1,\dots,m+1$. Then
\begin{align*}
	&\gamma T(x_j^{(1)},\dots,x_j^{(m+1)}) - \gamma T(x_1,\dots,x_{m+1}) \\
	&= \gamma(x_j^{(m+1)} - x_{m+1})\, T(x_j^{(1)},\dots,x_j^{(m)}) \\
	&\quad + \gamma(x_{m+1})\big(T(x_j^{(1)},\dots,x_j^{(m)}) - T(x_1,\dots,x_m)\big).
\end{align*}
Since $(T(x_j^{(1)},\dots,x_j^{(m)}))$ is bounded and $x_j^{(m+1)} \overset{w}{\to} x_{m+1}$, it follows that
\begin{align*}
	&\|\gamma T(x_j^{(1)},\dots,x_j^{(m+1)}) - \gamma T(x_1,\dots,x_{m+1})\| \\
	&\le |\gamma(x_j^{(m+1)} - x_{m+1})|\, \|T(x_j^{(1)},\dots,x_j^{(m)})\| \\
	&\quad + |\gamma(x_{m+1})|\, \|T(x_j^{(1)},\dots,x_j^{(m)}) - T(x_1,\dots,x_m)\| \to 0
\end{align*}
as $j \to \infty$. The norm estimate is immediate.

The conclusion follows from \cite[Remark 2.3(3) and Proposition 2.4]{RST22} together with Proposition \ref{DPS}.
	
\end{proof}

\begin{remark}\label{RSCSC}
	We present an example showing that this class is not strongly coherent, according to \cite{RST22}. Let $E_1,\dots,E_m$ and $F$ be Banach spaces, let $\varphi_i \in E_i'$, $i=1,\dots,m$, and let $y \in F$ be nonzero. Choose $(a_1,\dots,a_m) \in E_1 \times \cdots \times E_m$ such that $a_i \notin \ker \varphi_i$ for all $i$.
	
	Define $T \colon E_1 \times \cdots \times E_m \to F$ by
	\begin{equation*}
		T(x_1,\dots,x_m) = \varphi_1(x_1)\cdots \varphi_m(x_m)\, y,
	\end{equation*}
	and let $Q \colon \ell_2 \times \ell_2 \to \mathbb{K}$ be given by
	\begin{equation*}
		Q\big((x_j)_{j=1}^{\infty}, (y_j)_{j=1}^{\infty}\big) = \sum_{j=1}^{\infty} x_j y_j.
	\end{equation*}
	
	Let $(x_j^{(i)})_{j=1}^{\infty} \subset E_i$ be such that $x_j^{(i)} \overset{w}{\to} a_i$ for $i=1,\dots,m$. Since $e_j \overset{w}{\to} 0$ in $\ell_2$, we have
	\begin{align*}
		QT(x_j^{(1)},\dots,x_j^{(m)},e_j,e_j)
		&= Q(e_j,e_j)\, T(x_j^{(1)},\dots,x_j^{(m)}) \\
		&= \varphi_1(x_j^{(1)})\cdots \varphi_m(x_j^{(m)})\, y \\
		&\longrightarrow \varphi_1(a_1)\cdots \varphi_m(a_m)\, y \neq 0.
	\end{align*}
	Thus,
	\begin{equation*}
		QT(x_j^{(1)},\dots,x_j^{(m)},e_j,e_j) \nrightarrow QT(a_1,\dots,a_m,0,0),
	\end{equation*}
	and therefore $QT \notin \mathcal{L}_{wsc}(E_1,\dots,E_m,\ell_2,\ell_2;F)$, even though $T \in \mathcal{L}_{wsc}(E_1,\dots,E_m;F)$.
\end{remark}

We denote by $\mathcal{L_{DP}^{\mathrm{(CH1)}}}$ the class of all Dunford–Pettis multilinear operators between Banach spaces satisfying the condition (CH1) (see \cite[Definition 3.2]{PR14} and \cite[Definition 2.1]{RST22}).

\begin{proposition}\label{DPCH1}
	Let $E_1,\dots,E_m$ and $F$ be Banach spaces. Then
	\begin{equation*}
		\mathcal{L_{DP}^{\mathrm{(CH1)}}} = \mathcal{L}_{wsc}.
	\end{equation*}
\end{proposition}

\begin{proof}
	Let $E_1,\dots,E_m$ and $F$ be Banach spaces. It follows from item (a) of Remark \ref{RIRC} and Theorem \ref{DPcoercomp} that
	\begin{equation*}
		\mathcal{L}_{wsc} \subset \mathcal{L_{DP}^{\mathrm{(CH1)}}}.
	\end{equation*}
	The reverse inclusion is a direct application of condition $(CH1)$.
	
%	We now prove the reverse inclusion. Let $T \in \mathcal{L_{DP}^{\mathrm{(CH1)}}}(E_1,\dots,E_m;F)$. We need to show  $T \in \mathcal{L}_{wsc}(E_1,\dots,E_m;F)$. Let $(x_j^{(i)})_{j=1}^{\infty} \subset E_i$ be such that
%	\begin{equation*}
%		x_j^{(i)} \overset{w}{\to} a_i \in E_i,
%	\end{equation*}
%	for all $i=1,\dots,m$. Then $(x_j^{(i)} - a_i) \overset{w}{\to} 0$ in $E_i$ for all $i=1,\dots,m$.
	
%	For simplicity, we consider the case $m=3$; the general case follows analogously. We have
%	\begin{align*}
%		&T(x_j^{(1)},x_j^{(2)},x_j^{(3)}) - T(a_1,a_2,a_3) \\
%		&= T(x_j^{(1)}-a_1, x_j^{(2)}-a_2, x_j^{(3)}-a_3) \\
%		&\quad + T_{a_2,a_3}(x_j^{(1)}-a_1) + T_{a_1,a_3}(x_j^{(2)}-a_2) + T_{a_1,a_2}(x_j^{(3)}-a_3) \\
%		&\quad + T_{a_3}(x_j^{(1)}-a_1, x_j^{(2)}-a_2) + T_{a_2}(x_j^{(1)}-a_1, x_j^{(3)}-a_3) + T_{a_1}(x_j^{(2)}-a_2, x_j^{(3)}-a_3).
%	\end{align*}
	
%	By property (CH1), we have
%	\begin{equation*}
%		T_{a_1} \in \mathcal{L_{DP}^{\mathrm{(CH1)}}}(E_2,E_3;F), \quad
%		T_{a_2} \in \mathcal{L_{DP}^{\mathrm{(CH1)}}}(E_1,E_3;F), \quad
%		T_{a_3} \in \mathcal{L_{DP}^{\mathrm{(CH1)}}}(E_1,E_2;F),
%	\end{equation*}
%	and hence
%	\begin{equation*}
%		T_{a_1,a_2} \in \mathcal{DP}(E_3;F), \quad
%		T_{a_1,a_3} \in \mathcal{DP}(E_2;F), \quad
%		T_{a_2,a_3} \in \mathcal{DP}(E_1;F).
%	\end{equation*}
	
%	So, each term in the above decomposition converges to zero as $j \to \infty$, and therefore
%	\begin{equation*}
%		\|T(x_j^{(1)},x_j^{(2)},x_j^{(3)}) - T(a_1,a_2,a_3)\| \to 0.
%	\end{equation*}
%	Thus, $T \in \mathcal{L}_{wsc}(E_1,\dots,E_m;F)$.
\end{proof}

A natural question is under what conditions every continuous multilinear operator is weakly sequentially continuous. The following proposition provides an answer to this question.

%The next result provides a condition under which every multilinear operator between Banach spaces is Dunford–Pettis at every point.

\begin{proposition}\label{DPEVS}
	Let $E_1,\dots,E_m$ and $F$ be Banach spaces. If each $E_i$ has the Schur property for $i=1,\dots,m$, then every continuous multilinear operator from $E_1 \times \cdots \times E_m$ into $F$ is weakly sequentially continuous. In particular,
	\begin{equation*}
		\mathcal{L}_{wsc}(E_1,\dots,E_m;F) = \mathcal{L}(E_1,\dots,E_m;F).
	\end{equation*}
\end{proposition}

\begin{corollary}
	Let $E_1,\dots,E_m$ and $F$ be Banach spaces. If each $E_i$ has the Schur property for $i=1,\dots,m$, then
	\[
	\mathcal{L_{DP}}(E_1,\dots,E_m;F) = \mathcal{L}_{wsc}(E_1,\dots,E_m;F).
	\]
\end{corollary}

Let us recall that, given Banach spaces $E_1,\dots,E_m$, we denote their projective tensor product by $E_1 \hat{\otimes}_{\pi} \cdots \hat{\otimes}_{\pi} E_m$ (see \cite{R02} for details). As noted in \cite[Section 4]{BT15}, for any Banach spaces $E_1,\dots,E_m$ and $F$, if $E_1 \hat{\otimes}_{\pi} \cdots \hat{\otimes}_{\pi} E_m$ has the Schur property, then
\begin{equation*}
	\mathcal{CC} \circ \mathcal{L}(E_1,\dots,E_m;F) = \mathcal{L}(E_1,\dots,E_m;F).
\end{equation*}
Therefore, we readily obtain the following result.

\begin{proposition}
	Let $E_1,\dots,E_m$ and $F$ be Banach spaces. If $E_1 \hat{\otimes}_{\pi} \cdots \hat{\otimes}_{\pi} E_m$ also has the Schur property, then
	\begin{equation*}
		\mathcal{CC} \circ \mathcal{L}(E_1,\dots,E_m;F)
		= \mathcal{L}_{wsc}(E_1,\dots,E_m;F).
	\end{equation*}
\end{proposition}

\begin{proof}
	It's easy to see that if $E_1 \hat{\otimes}_{\pi} \cdots \hat{\otimes}_{\pi} E_m$ then, each $E_i$, $i=\{1,\dots, m\}$ also has the Schur property. So,
	\begin{equation*}
		\mathcal{L}_{wsc}(E_1,\dots,E_m;F)
		= \mathcal{L}(E_1,\dots,E_m;F)
		= \mathcal{CC} \circ \mathcal{L}(E_1,\dots,E_m;F).
	\end{equation*}
\end{proof}

\begin{example}
Considering $E_1 = \cdots =E_m = \ell_1$, for every Banach space $F$ we have
\begin{equation*}
		\mathcal{L}_{wsc}(\ell_1,\dots,\ell_1;F)
		= \mathcal{L}(\ell_1,\dots,\ell_1;F)
		= \mathcal{CC} \circ \mathcal{L}(\ell_1,\dots,\ell_1;F),
	\end{equation*}
because, $\ell_1\widehat{\otimes}_{\pi}\cdots \widehat{\otimes}_{\pi} \ell_1 {\cong} \ell_1$ has the Schur property.
\end{example}

\section{Weakly Dunford-Pettis Multilinear Operators}

In this section, we define the class of weakly Dunford-Pettis multilinear operators (at the origin and everywhere). We begin by recalling the linear case (see \cite{AB85}) and then extend the definition to the multilinear setting. Among the various results concerning this class, we show that it is not a hyperideal (see \cite{BT15}), despite being a multi-ideal of multilinear operators, nor does it satisfy the coherence property (see \cite{PR14}), since, in general, condition $(CH1)$ is not met.

%We begin by defining the linear case, which will be useful for showing that the class of weakly Dunford-Pettis linear operators is not a hyper-ideal (see \cite{BT15}) and that the coherence and compatibility condition (CH1) fails (see \cite{PR14}).

\begin{definition}
	Let $E$ and $F$ be Banach spaces and let $T \in \mathcal{L}(E;F)$. We say that $T$ is weakly Dunford-Pettis if, for every pair of sequences $(x_j)_{j=1}^{\infty} \subset E$ and $(f_j)_{j=1}^{\infty} \subset F'$ such that $x_j \overset{w}{\to} 0$ and $f_j \overset{w}{\to} 0$, we have
	\begin{equation*}
		f_j\big(T(x_j)\big) \longrightarrow 0.
	\end{equation*}
\end{definition}

We denote by $\mathcal{WDP}(E;F)$ the class of all weakly Dunford-Pettis linear operators from $E$ into $F$. It is straightforward to verify that $(\mathcal{WDP},\|\cdot\|)$ is a Banach operator ideal.

%We will say that a Banach space has the Dunford-Pettis property if the identity, $Id_E\colon E \rightarrow E$ given by $Id_X(x) = x$ for all $x \in E$, is a weakly Dunford-Pettis application.

\begin{definition}
	Let $E_1,\dots,E_m$ and $F$ be Banach spaces, and let $T \in \mathcal{L}(E_1,\dots,E_m;F)$. We say that $T$ is a weakly Dunford-Pettis $m$-linear operator if, for every family of sequences $(x_j^{(i)})_{j=1}^{\infty} \subset E_i$, $i=1,\dots,m$, and every sequence $(f_j)_{j=1}^{\infty} \subset F'$ such that $x_j^{(i)} \overset{w}{\to} 0$ for all $i=1,\dots,m$ and $f_j \overset{w}{\to} 0$, we have
	\begin{equation*}
		f_j\big(T(x_j^{(1)},\dots,x_j^{(m)})\big) \longrightarrow 0.
	\end{equation*}
\end{definition}

We denote by $\mathcal{L_{WDP}}$ the class of all weakly Dunford-Pettis multilinear operators.

\begin{proposition}\label{DPWDP}
	Let $E_1,\dots,E_m$ and $F$ be Banach spaces. Every Dunford--Pettis $m$-linear operator is weakly Dunford--Pettis.
\end{proposition}

\begin{proof}
	Let $T \in \mathcal{L_{DP}}(E_1,\dots,E_m;F)$. Consider $(x_j^{(i)})_{j=1}^{\infty} \subset E_i$, $i=1,\dots,m$, and $(f_j)_{j=1}^{\infty} \subset F'$ be such that $x_j^{(i)} \overset{w}{\to} 0$ for all $i=1,\dots,m$ and $f_j \overset{w}{\to} 0$.
	
	Since $T$ is Dunford--Pettis, we have
	\begin{equation*}
		\|T(x_j^{(1)},\dots,x_j^{(m)})\| \longrightarrow 0.
	\end{equation*}
	As, every weakly convergent sequence is bounded, there exists $K>0$ such that $\|f_j\| \le K$ for all $j$. Therefore,
	\begin{equation*}
		|f_j(T(x_j^{(1)},\dots,x_j^{(m)}))|
		\le \|f_j\| \, \|T(x_j^{(1)},\dots,x_j^{(m)})\|
		\le K \, \|T(x_j^{(1)},\dots,x_j^{(m)})\|.
	\end{equation*}
	Hence,
	\begin{equation*}
		f_j(T(x_j^{(1)},\dots,x_j^{(m)})) \longrightarrow 0,
	\end{equation*}
	and thus $T$ is weakly Dunford--Pettis.
\end{proof}

Following the ideas presented by Ribeiro, Santos and Torres in \cite[Proposition 2.8]{RST22}, we can easily see that $(\mathcal{L_{WDP}}, \|\cdot\|)$, endowed with the usual norm, is a Banach multi-ideal that generalizes the Banach operator ideal $\mathcal{WDP}$.

	It is important to emphasize that, in general, the inclusion in Proposition \ref{DPWDP} is strict, since there exist weakly Dunford-Pettis multilinear operators that are not Dunford-Pettis, as will be shown below.

%%%%%%%%%%%%%%% EXEMPLO RESUMIDO E COMENTADO %%%%%%%%%%%%%%%%%%%%
%\begin{example}
%	Consider the mapping $T\colon c_0 \times c_0 \to c_0$ defined by
%	\[
%	T\big((x_j)_{j=1}^{\infty}, (y_j)_{j=1}^{\infty}\big) = (x_j y_j)_{j=1}^{\infty}.
%	\]
%	Since $e_j \overset{w}{\to} 0$ in $c_0$ and
%	\[
%	\|T(e_j,e_j)\|_{\infty} = \|e_j\|_{\infty} = 1,
%	\]
%	it follows that $T$ is not a Dunford-Pettis bilinear mapping. 
%	
%	On the other hand, using the fact that $c_0' = \ell_1$ and that $\ell_1$ has the Schur property, it is straightforward to verify that $T \in \mathcal{L}_{WDP}(c_0,c_0;c_0)$.
%\end{example}

\begin{example}
	Consider the mapping $T\colon c_0 \times c_0 \to c_0$ defined by
	\[
	T\big((x_j)_{j=1}^{\infty}, (y_j)_{j=1}^{\infty}\big) = (x_j y_j)_{j=1}^{\infty}.
	\]
	Since $e_j \overset{w}{\to} 0$ in $c_0$ and
	\[
	\|T(e_j,e_j)\|_{\infty} = \|e_j\|_{\infty} = 1,
	\]
	it follows that $T$ is not a Dunford--Pettis bilinear mapping. On the other hand, we claim that $T \in \mathcal{L_{WDP}}(c_0,c_0;c_0)$. Indeed, let $(x_j)_{j=1}^{\infty}$ and $(y_j)_{j=1}^{\infty}$ be sequences in $c_0$ such that $x_j \overset{w}{\to} 0$ and $y_j \overset{w}{\to} 0$, and let $(f_j)_{j=1}^{\infty} \subset \ell_1 = c_0'$ be such that $f_j \overset{w}{\to} 0$. Since $\ell_1$ has the Schur property, it follows that $\|f_j\|_1 \to 0$. Moreover, the sequences $(x_j)$ and $(y_j)$ are bounded, thus
	\[
	|f_j(T(x_j,y_j))|
	\le \|f_j\|_1 \, \|T(x_j,y_j)\|_{\infty}
	\le \|f_j\|_1 \, \|x_j\|_{\infty} \|y_j\|_{\infty} \longrightarrow 0.
	\]
	Hence, $T$ is weakly Dunford--Pettis.
\end{example}

As observed above, the class $(\mathcal{L_{WDP}}, \|\cdot\|)$ is a Banach multi-ideal and therefore enjoys the usual properties of such ideals. It is natural to ask whether this class also satisfies the hyper-ideal property introduced by Botelho and Torres in \cite{BT15}, that is, whether $(\mathcal{L_{WDP}}, \|\cdot\|)$ forms a hyper-ideal of multilinear mappings.

\begin{remark}
	Although $(\mathcal{L_{WDP}}, \|\cdot\|)$ is a Banach multi-ideal, it is not a hyper-ideal of multilinear mappings. Indeed, consider the mapping
	\[
	T\colon \ell_2 \times \ell_2 \to \ell_1, \quad
	T\big((x_j)_{j=1}^{\infty}, (y_j)_{j=1}^{\infty}\big) = (x_j y_j)_{j=1}^{\infty},
	\]
	as in Example 2.9 of \cite{RST22}. It is straightforward to verify that $T \notin \mathcal{L_{WDP}}(\ell_2,\ell_2;\ell_1)$. On the other hand, by Schur's theorem (see \cite[Theorem 1.7]{DJT95}) and the fact that
	\[
	\mathcal{CC}(\ell_1) =  \mathcal{WDP}(\ell_1),
	\]
	we have $Id_{\ell_1} \in \mathcal{WDP}(\ell_1)$. If $(\mathcal{L_{WDP}}, \|\cdot\|)$ were a hyper-ideal, then
	\[
	T = Id_{\ell_1} \circ T \in \mathcal{L_{WDP}}(\ell_2,\ell_2;\ell_1),
	\]
	which is a contradiction.
\end{remark}
\begin{proposition}
Let $E_1,\dots,E_m$ and $F$ be Banach spaces. If $F^{'}$ has the Schur property then
\begin{equation*}
\mathcal{L_{WDP}}(E_1,\dots, E_m; F) = \mathcal{L}(E_1,\dots, E_m; F).
\end{equation*}
\end{proposition}

\begin{proof}
 Let $(x_j^{(i)})_{j=1}^{\infty} \subset E_i$, $i=1,\dots,m$, and $(f_j)_{j=1}^{\infty} \subset F^{'}$ be such that $x_j^{(i)} \overset{w}{\to} 0$ for all $i=1,\dots,m$ and $f_j \overset{w}{\to} 0$. Since, $F^{'}$ has the Schur property then $\|f_j\| \to 0$, as $j \to \infty$. So,
\begin{equation*}
\|f_j(T(x_j^{(1)},\dots, x_j^{(m)})\| \le \|f_j\| \|T\| \|x_j^{(1)}\|\cdots \|x_j^{(m)}\| \to 0
\end{equation*}
as $j \to \infty$, because $(x_j^{(i)})_{j=1}^{\infty}$ is bounded sequences in $E_i$, $i=1,\dots, m$.
\end{proof}

\begin{definition}
	Let $m \in \mathbb{N}$, and let $E$ and $F$ be Banach spaces. A mapping $P \in \mathcal{P}(^mE;F)$ is said to be a weakly Dunford-Pettis $m$-homogeneous polynomial if, for every sequence $(x_j)_{j=1}^{\infty} \subset E$ and every sequence $(f_j)_{j=1}^{\infty} \subset F'$ such that $x_j \overset{w}{\to} 0$ and $f_j \overset{w}{\to} 0$, we have
	\begin{equation*}
		f_j\big(P(x_j)\big) \longrightarrow 0.
	\end{equation*}
\end{definition}

We denote by $\mathcal{P_{WDP}}$ the class of all weakly Dunford-Pettis homogeneous polynomials.

It is possible to relate the class of weakly Dunford--Pettis multilinear mappings to that of weakly Dunford-Pettis homogeneous polynomials through the following result.
%whose proof follows the same ideas as the proof of condition $(CH5)$ in Theorem \ref{DPcoercomp}.

Before proving this result, we recall an important notion in multilinear theory. Given a Banach multi-ideal $(\mathcal{M}, \|\cdot\|_{\mathcal{M}})$, the associated Banach ideal of homogeneous polynomials is defined by
\begin{equation*}
	\mathcal{P}_{\mathcal{M}} = \{P \in \mathcal{P} \; ; \; \check{P} \in \mathcal{M}\},
\end{equation*}
endowed with the norm
\begin{equation*}
	\|P\|_{\mathcal{P}_{\mathcal{M}}} := \|\check{P}\|_{\mathcal{M}}.
\end{equation*}

\begin{proposition}\label{CHSCHCO}
	Let $m \in \mathbb{N}$, and let $E$ and $F$ be Banach spaces. Then $P \in \mathcal{P_{WDP}}(^mE;F)$ if and only if $\check{P} \in \mathcal{L_{WDP}}(^mE;F)$.
\end{proposition}

\begin{proof}
	Let $m \in \mathbb{N}$, $E$ and $F$ be Banach Spaces. Suppose $\check{P} \in \mathcal{L_{WDP}}(^mE; F)$ and consider
%vamo mostrar que $P \in \mathcal{P_{WDP}}(^mE; F)$. 
$(x_j)_{j=1}^{\infty} \subset E$ and $(f_j)_{j=1}^{\infty} \subset F'$ such that $x_j^{(i)} \overset{w}{\longrightarrow} 0$ e $f_j \overset{w}{\longrightarrow} 0$ then
	\begin{equation*}
		|f_j\left(P(x_j)\right)| = |f_j\left(\check{P}(x_j,\dots, x_j)\right)| \longrightarrow 0.
	\end{equation*} 
Suppose now $P \in \mathcal{P_{WDP}}(^mE; F)$ and consider $(x_j^{(i)})_{j=1}^{\infty} \subset E$ and $(f_j)_{j=1}^{\infty} \subset F'$ with $x_j^{(i)} \overset{w}{\longrightarrow} 0$ and $f_j \overset{w}{\longrightarrow} 0$, for all $i=1,\dots, m$. Since, for all $\epsilon_i = \pm 1$ we have
\begin{equation*}
	\epsilon_1 x_j^{(1)} + \cdots + \epsilon_m x_j^{(m)} \overset{w}{\longrightarrow} 0.
\end{equation*}
So,
\begin{align*}
	\left|f_j\left(\check{P}\left(x_j^{(1)},\dots, x_j^{(m)}\right)\right)\right| &= \left|f_j\left(\frac{1}{2^mm!}\check{P}(\epsilon_1x_j^{(1)} + \cdots + \epsilon_mx_j^{(m)})^m\right)\right|\\
	&= \frac{1}{2^mm!}\left|f_j\left(\check{P}(\epsilon_1x_j^{(1)} + \cdots + \epsilon_mx_j^{(m)})^m\right)\right|\\
	&= \frac{1}{2^mm!}\left|f_j\left(P(\epsilon_1x_j^{(1)} + \cdots + \epsilon_mx_j^{(m)})\right)\right| \longrightarrow 0.
\end{align*}
as $j \rightarrow \infty$. Hence, $\check{P} \in \mathcal{L_{WDP}}(^mE; F)$.
\end{proof}

As an immediate consequence of Proposition \ref{CHSCHCO}, we have
\begin{equation*}
	\mathcal{P_{WDP}}(^mE;F) = \mathcal{P}_{\mathcal{L_{WDP}}}(^mE;F)
\end{equation*}
for every $m \in \mathbb{N}$ and all Banach spaces $E$ and $F$.

Another immediate consequence of Proposition \ref{CHSCHCO} is that the pair $(\mathcal{P_{WDP}}(^mE;F), \|\cdot\|)$ is a Banach ideal of homogeneous polynomials.

The result in Proposition \ref{CHSCHCO} corresponds to condition $(CH5)$ and $(CP5)$ in the definitions of coherent and compatible, and strongly coherent and compatible pairs (see \cite{PR14, RST22}). This naturally leads to the question of whether the sequence of pairs $\left(\mathcal{L_{WDP}}^n, \mathcal{P_{WDP}}^n\right)_{n=1}^{\infty}$ is strongly coherent and strongly compatible with $\mathcal{WDP}$.

\begin{proposition}
	Let $m,n \in \mathbb{N}$, let $E_1,\dots,E_{m+n}$ and $F$ be Banach spaces. If $T \in \mathcal{L_{WDP}}(E_1,\dots,E_m;F)$ and $Q \in \mathcal{L}(E_{m+1},\dots,E_{m+n})$, then
	\[
	QT \in \mathcal{L_{WDP}}(E_1,\dots,E_{m+n};F).
	\]
\end{proposition}

\begin{proof}
	Let $(x_j^{(i)})_{j=1}^{\infty} \subset E_i$, $i=1,\dots,m+n$, and $(f_j)_{j=1}^{\infty} \subset F'$ be such that $x_j^{(i)} \overset{w}{\to} 0$ for all $i=1,\dots,m+n$ and $f_j \overset{w}{\to} 0$.
	
	Since, $Q$ is continuous, the sequence $\big(Q(x_j^{(m+1)},\dots,x_j^{(m+n)})\big)_{j=1}^{\infty}$ is bounded. Moreover, since $T \in \mathcal{L_{WDP}}(E_1,\dots,E_m;F)$, we have
	\[
	f_j\big(T(x_j^{(1)},\dots,x_j^{(m)})\big) \longrightarrow 0.
	\]
	Therefore,
	\[
	\big|f_j\big(QT(x_j^{(1)},\dots,x_j^{(m+n)})\big)\big|
	= \big|Q(x_j^{(m+1)},\dots,x_j^{(m+n)})\big|
	\, \big|f_j\big(T(x_j^{(1)},\dots,x_j^{(m)})\big)\big|
	\longrightarrow 0,
	\]
	which shows that $QT \in \mathcal{L_{WDP}}(E_1,\dots,E_{m+n};F)$.
	
	The norm estimate
	\[
	\|QT\| \le \|Q\| \, \|T\|
	\]
	is immediate.
\end{proof}

Using similar arguments, we obtain the following result.

\begin{proposition}\label{CH3CHCO}
	Let $m,n \in \mathbb{N}$ and let $E$ and $F$ be Banach spaces. If $P \in \mathcal{P_{WDP}}(^mE;F)$ and $Q \in \mathcal{P}(^nE)$, then
	\[
	QP \in \mathcal{P_{WDP}}(^{m+n}E;F).
	\]
\end{proposition}

\begin{remark}\label{WDP(CH1)}
	Although the class $\mathcal{L_{WDP}}$ satisfies conditions $(CH3)$ and $(CH5)$, as shown in Propositions \ref{CHSCHCO} and \ref{CH3CHCO}, it does not satisfy condition $(CH1)$.	
	Indeed, let $E$ be a Banach space and let $a \in E \setminus \{0\}$. By the Hahn--Banach Theorem, there exists $\varphi \in E'$ such that $|\varphi(a)| = 1$. Define $T\colon E \times \ell_2 \to \ell_2$ by
	\[
	T(x,(x_j)_{j=1}^{\infty}) = (\varphi(x)x_j)_{j=1}^{\infty}.
	\]
	It is immediate that $T$ is a continuous bilinear mapping.
	
	Let $x_j \overset{w}{\to} 0$ in $E$, $y_j \overset{w}{\to} 0$ in $\ell_2$, and $f_j \overset{w}{\to} 0$ in $(\ell_2)'$. Then $\varphi(x_j) \to 0$, while the sequences $(y_j)_{j=1}^{\infty}$ and $(f_j)_{j=1}^{\infty}$ are bounded in $\ell_2$ and $(\ell_2)'$, respectively. Therefore,
	\[
	|f_j(T(x_j,y_j))| \longrightarrow 0,
	\]
	which shows that $T \in \mathcal{L_{WDP}}(E,\ell_2;\ell_2)$.
	
	However, $T_a \notin \mathcal{WDP}(\ell_2;\ell_2)$. Indeed, let $(e_j)_{j=1}^{\infty}$ be the canonical basis of $\ell_2$. Since $(\ell_2)'$ is isometrically isomorphic to $\ell_2$, we may identify each $f_j \in (\ell_2)'$ with $e_j$. Writing $e_j = (\delta_{n,j})_{n=1}^{\infty}$, we obtain
	\[
	|f_j(T_a(e_j))| = |f_j(T(a,e_j))|
	= \left|f_j\big((\varphi(a)\delta_{n,j})_{n=1}^{\infty}\big)\right|
	= |\varphi(a)| \|e_j\|_2
	= |\varphi(a)| = 1,
	\]
	for all $j \in \mathbb{N}$.
	
	Thus, $(T_a(e_j))_{j=1}^{\infty}$ does not converge to zero in $\ell_2$, and hence $T_a$ does not satisfy condition $(CH1)$.
\end{remark}

It is worth noting that, by Proposition \ref{DPEVS}, if $E_i$ has the Schur property for all $i=1,\dots,m$, then
\[
\mathcal{L}_{wsc}(E_1,\dots,E_m;F)
= \mathcal{L_{DP}}(E_1,\dots,E_m;F)
= \mathcal{L_{WDP}}(E_1,\dots,E_m;F).
\]

Let us now introduce the concept of weakly Dunford–Pettis multilinear operators at every point.

\begin{definition}\label{WDPML}
	Let $m \in \mathbb{N}$, let $E_1,\dots,E_m$ and $F$ be Banach spaces, and let $T \in \mathcal{L}(E_1,\dots,E_m;F)$. We say that $T$ is a weakly Dunford–Pettis $m$-linear operator at $(a_1,\dots,a_m) \in E_1 \times \cdots \times E_m$ if for every sequence $(x_j^{(i)})_{j=1}^{\infty} \subset E_i$, $i=1,\dots,m$, and every sequence $(f_j)_{j=1}^{\infty} \subset F'$ such that
	\[
	x_j^{(i)} \overset{w}{\longrightarrow} a_i \quad \text{for all } i=1,\dots,m,
	\quad \text{and} \quad
	f_j \overset{w}{\longrightarrow} f \in F',
	\]
	we have
	\[
	f_j\big(T(x_j^{(1)},\dots,x_j^{(m)})\big)
	\longrightarrow
	f\big(T(a_1,\dots,a_m)\big)
	\]
	as $j \to \infty$.
\end{definition}

When $T$ is a weakly Dunford–Pettis multilinear operator at every point of $E_1 \times \cdots \times E_m$, we say that $T$ is weakly Dunford–Pettis at every point and denote the corresponding class by $\mathcal{L_{WDP}^{\mathrm{ev}}}$.

\begin{definition}\label{WDPPH}
	Let $m \in \mathbb{N}$, let $E$ and $F$ be Banach spaces, and let $P \in \mathcal{P}(^mE;F)$. We say that $P$ is a weakly Dunford–Pettis $m$-homogeneous polynomial at $a \in E$ if for every sequence $(x_j)_{j=1}^{\infty} \subset E$ and every sequence $(f_j)_{j=1}^{\infty} \subset F'$ such that
	\[
	x_j \overset{w}{\longrightarrow} a
	\quad \text{and} \quad
	f_j \overset{w}{\longrightarrow} f \in F',
	\]
	we have
	\[
	f_j\big(P(x_j)\big) \longrightarrow f\big(P(a)\big)
	\]
	as $j \to \infty$.
\end{definition}

When $P$ is a weakly Dunford–Pettis homogeneous polynomial at every point of $E$, we say that $P$ is weakly Dunford–Pettis at every point, and we denote the class of all such polynomials by $\mathcal{P_{WDP}^{\mathrm{ev}}}$.

We now present a characterization of weakly Dunford-Pettis multilinear operators at every point, analogous to the Dunford-Pettis case.

\begin{proposition}
	Let $E_1,\dots, E_m$ and $F$ be Banach spaces and $T \in \mathcal{L}(E_1,\dots, E_m; F)$. Then, $T \in \mathcal{L_{WDP}^{\text{ev}}}(E_1,\dots, E_m; F)$ if, and only if, for every $(a_1,\dots, a_m) \in E_1\times\cdots\times E_m$, every family of weakly null sequences $(x_j^{(i)})_{j=1}^{\infty} \subset E_i$, $i=1,\dots,m$, and every sequence $(f_j)_{j=1}^{\infty} \subset F'$ with $f_j \overset{w}{\longrightarrow} f$, we have
	\begin{equation*}
		f_j\big(T(x_j^{(1)} + a_1,\dots, x_j^{(m)} + a_m)\big)
		\longrightarrow
		f\big(T(a_1,\dots, a_m)\big).
	\end{equation*}
\end{proposition}

%\begin{proof}
%	Suppose first that $T \in \mathcal{L_{WDP}^{\text{ev}}}(E_1,\dots, E_m; F)$.
%	
%	Let $(a_1,\dots,a_m)\in E_1\times\cdots\times E_m$, let $(x_j^{(i)})$ be weakly null sequences in $E_i$, and let $(f_j) \subset F'$ be such that $f_j \overset{w}{\longrightarrow} f$. Then
%	\[
%	x_j^{(i)} + a_i \overset{w}{\longrightarrow} a_i
%	\]
%	for all $i=1,\dots,m$. Since $T$ is weakly Dunford-Pettis at every point, it follows that
%	\[
%	f_j\big(T(x_j^{(1)} + a_1,\dots, x_j^{(m)} + a_m)\big)
%	\longrightarrow
%	f\big(T(a_1,\dots, a_m)\big).
%	\]
%	Conversely, suppose the condition holds. Let $(x_j^{(i)}) \subset E_i$ and $(f_j) \subset F'$ be such that
%	\[
%	x_j^{(i)} \overset{w}{\longrightarrow} a_i
%	\quad \text{and} \quad
%	f_j \overset{w}{\longrightarrow} f.
%	\]
%	Define $y_j^{(i)} = x_j^{(i)} - a_i$, so that $y_j^{(i)} \overset{w}{\longrightarrow} 0$. Then
%	\[
%	f_j\big(T(x_j^{(1)},\dots,x_j^{(m)})\big)
%	=
%	f_j\big(T(y_j^{(1)} + a_1,\dots,y_j^{(m)} + a_m)\big).
%	\]
%	By hypothesis,
%	\[
%	f_j\big(T(y_j^{(1)} + a_1,\dots,y_j^{(m)} + a_m)\big)
%	\longrightarrow
%	f\big(T(a_1,\dots,a_m)\big).
%	\]
%	Therefore, $T \in \mathcal{L_{WDP}^{\text{ev}}}(E_1,\dots,E_m;F)$.
%\end{proof}

%This characterization shows that weakly Dunford-Pettis multilinear operators at every point are completely determined by their behavior on weakly null perturbations.

\begin{remarks}\label{RIRC2}
\begin{enumerate}
\item[$(a)$] It follows immediately from the definitions above that $\mathcal{L_{WDP}^{\mathrm{ev}}} \subset \mathcal{L_{WDP}}$ and $\mathcal{P_{WDP}^{\mathrm{ev}}} \subset \mathcal{P}_{WDP}$.
\item[$(b)$] It is well established that for any Banach spaces $E_1, \dots, E_m$ and $F$, the collection $\mathcal{L_{WDP}^{\mathrm{ev}}}(E_1,\dots,E_m;F)$ is a closed subset of $\mathcal{L_{WDP}}(E_1,\dots,E_m;F)$.
\end{enumerate}
\end{remarks}

It is important to note the inclusion 
\begin{equation*}
\mathcal{L}_{WDP}^{ev} \subset \mathcal{L}_{WDP}
\end{equation*}
of item $(a)$ of Remarkes \ref{RIRC2}, in generale, is strict, as shown by the following example. Consider $T\colon \ell_1 \times \ell_2 \to \ell_2$, given by $T(x, y) = e_1^*(x)y$, where  $e_1^* \in \ell_1^{'}$ where $e_1^*((x_j)_{j=1}^{\infty}) = x_1$, for every $(x_j)_{j=1}^{\infty} \in \ell_1$. It is immediate see $T \in \mathcal{L}_{WDP}(\ell_1, \ell_2; \ell_2)$. On other hand, considering $x_j = e_1 \in \ell_1$ and $y_j = e_j \in \ell_2$ and $f_j = e_j \in \ell_2 \cong \ell_2^{'}$, we have $x_j \xrightarrow{w} e_1$ in $\ell_1$, $y_j \xrightarrow{w} 0$ in $\ell_2$ and $f_j \xrightarrow{w} 0$ in $\ell_2^{'}$. So,
\begin{equation*}
|f_j(T(e_1, e_j))| = |f_j(e_j)| = 1 \not\rightarrow 0.
\end{equation*}
Hence, $T\notin \mathcal{L}_{WDP}^{ev}(\ell_1, \ell_2; \ell_2)$.

Using arguments analogous to those in Section \ref{DPML}, the following results can be established.

\begin{proposition}
	\begin{itemize}\
		\item[$(a)$] The class $(\mathcal{L_{WDP}^{\mathrm{ev}}}, \|\cdot\|)$ is a symmetric Banach multi-ideal.
		\item[$(b)$] The class $(\mathcal{P_{WDP}^{\mathrm{ev}}}, \|\cdot\|)$ is a Banach ideal of homogeneous polynomials.
		\item[$(c)$] The sequence $(\mathcal{L_{WDP}^{\mathrm{ev,n}}}, \mathcal{P_{WDP}^{\mathrm{ev,n}}})_{n=1}^{\infty}$ is coherent and compatible with $\mathcal{WDP}$.
	\end{itemize}
\end{proposition}

It following the same line of reasoning as in Proposition \ref{DPCH1}, and denoting by $\mathcal{L_{WDP}^{\mathrm{(CH1)}}}$ the class of all weakly Dunford-Pettis multilinear operators between Banach spaces that satisfy condition $(CH1)$, we obtain the following result.

\begin{proposition}\label{WDPCH1}
	Let $E_1,\dots,E_m$ and $F$ be Banach spaces. Then,
	\[
	\mathcal{L_{WDP}^{\mathrm{(CH1)}}} = \mathcal{L_{WDP}^{\mathrm{ev}}}.
	\]
\end{proposition}

%\begin{remark}
%	\begin{itemize}
%		\item[(a)] A classe $(\mathcal{L_{WDP}^{\text{ev}}}, \|\cdot\|)$ não é um Hiper-ideal, segundo \cite{BT15}, como mostra o seguinte exemplo (\textcolor{red}{apresentar exemplo}).
%		\item[(b)]  A sequência $(\mathcal{L_{WDP}^{\text{ev}}}_n, \mathcal{P_{WDP}^{\text{ev}}}_n)_{n=1}^{\infty}$ não é fortemente coerente, como mostra o seguinte exemplo (\textcolor{red}{apresentar exemplo}).
%	\end{itemize}
%\end{remark}

A simple computation allows us to see that 
\begin{equation*}
\mathcal{L}_{wsc}(E_1,\dots,E_m;F)
	\subset
	\mathcal{L_{WDP}^{\mathrm{ev}}}(E_1,\dots,E_m;F).
\end{equation*}
for every Banach spaces $E_1,\dots, E_m$ and $F$. The following example shows us that this inclusion is, in general, strict. Let $a \in c_0$ be a non-zero element, by Hahn-Banach there exists $\varphi \in c_0^{'}$ such that
\begin{equation*}
|\varphi(a)| = 1.
\end{equation*}
Now, consider $T\colon c_0 \times c_0 \to c_0$ given by, $T(x, y) = \varphi(x)y$. It is simple see that $T \in \mathcal{L}_{WDP}^{ev}(c_0, c_0; c_0)$, however, if $x_j = a$ and $y_j = e_j$ we have
\begin{equation*}
x_j \xrightarrow{w} a \text{ and } y_j \xrightarrow{w} 0
\end{equation*}
in $c_0$. So,
\begin{equation*}
\|T(a, e_j)\|_{\infty} =  \|\varphi(a) e_j\|_{\infty} = \|e_j\|_{\infty} = 1. 
\end{equation*} 
Hence, $T \notin \mathcal{L}_{wsc}(c_0, c_0; c_0)$.

In general, multiplying a continuous multilinear form by a weak Dunford-Pettis multilinear operator does not yield another weak Dunford-Pettis multilinear operator, as the example below shows.

\begin{example}
	 Let $E_1,\dots,E_m$ and $F$ be Banach spaces, and $a_i \in E_i$ be nonzero, $i=1,\dots, m$. By Hahn-Banach Theorem there exists $\varphi_i \in E_i^{'}$ such that $|\varphi(a_1)| = 1$. Consider $y \in F$ be nonzero, again, by the Hahn-Banach theorem, there is $f \in F^{'}$ such that $|f(y)| = 1$. Define
%let $\varphi_i \in E_i'$, $i=1,\dots,m$, and let $y \in F$ be nonzero. Choose $(a_1,\dots,a_m) \in E_1 \times \cdots \times E_m$ such that $a_i \notin \ker \varphi_i$ for all $i$.
	Define $T \colon E_1 \times \cdots \times E_m \to F$ by
	\begin{equation*}
		T(x_1,\dots,x_m) = \varphi_1(x_1)\cdots \varphi_m(x_m)\, y,
	\end{equation*}
	and let $Q \colon \ell_2 \times \ell_2 \to \mathbb{K}$ be given by
	\begin{equation*}
		Q\big((x_j)_{j=1}^{\infty}, (y_j)_{j=1}^{\infty}\big) = \sum_{j=1}^{\infty} x_j y_j.
	\end{equation*}
It is simple see that $T \in \mathcal{L_{WDP}^{\mathrm{ev}}}(E_1,\dots, E_m; F)$. Let $(x_j^{(i)})_{j=1}^{\infty} \subset E_i$ and $(f_j)_{j=1}^{\infty} \subset F^{'}$ be such that $x_j^{(i)} \overset{w}{\to} a_i$ for $i=1,\dots,m$ and $f_j = f$. Since $e_j \overset{w}{\to} 0$ in $\ell_2$ and $f_j \xrightarrow{w} f$, we have
	\begin{align*}
		&|f_j(QT(x_j^{(1)},\dots,x_j^{(m)},e_j,e_j)) - f(QT(a_1,\dots, a_m, 0, 0))|\\
		&= |Q(e_j,e_j)\, f(T(x_j^{(1)},\dots,x_j^{(m)}))| \\
		&= |\varphi_1(x_j^{(1)})|\cdots |\varphi_m(x_j^{(m)})|\, |f(y)| \\
		&\longrightarrow |\varphi_1(a_1)|\cdots |\varphi_m(a_m)|\, |f(y)| = 1 \neq 0.
	\end{align*}
	Thus,
	\begin{equation*}
		f_j(QT(x_j^{(1)},\dots,x_j^{(m)},e_j,e_j)) \nrightarrow f(QT(a_1,\dots,a_m,0,0)),
	\end{equation*}
	and therefore $QT \notin \mathcal{L_{WDP}^{\mathrm{ev}}}(E_1,\dots,E_m,\ell_2,\ell_2;F)$, even though $T \in \mathcal{L_{WDP}^{\mathrm{ev}}}(E_1,\dots, E_m; F)$.
\end{example}

This example leads us to ask the following question: Is there any condition we can impose on the Banach spaces $E_1, \dots, E_m$ or $F$ such that
\begin{equation*}
\mathcal{L}_{wsc}(E_1,\dots,E_m;F)
	=
	\mathcal{L_{WDP}^{\mathrm{ev}}}(E_1,\dots,E_m;F)?
\end{equation*}
The next result addresses this question.

\begin{theorem}\label{WDP=WDP}
	Let $E_1,\dots,E_m$ and $F$ be Banach spaces. If $F$ is reflexive, then
	\[
	\mathcal{L}_{wsc}(E_1,\dots,E_m;F)
	=
	\mathcal{L_{WDP}^{\mathrm{ev}}}(E_1,\dots,E_m;F).
	\]
\end{theorem}

\begin{proof}
	We already know from  that
	\[
	\mathcal{L}_{wsc}(E_1,\dots, E_m; F) \subset \mathcal{L_{WDP}^{\text{ev}}}(E_1,\dots, E_m; F)
	\]
	for any Banach spaces $E_1,\dots, E_m$ and $F$.
	
	To prove the reverse inclusion, let $T \in \mathcal{L_{WDP}^{\text{ev}}}(E_1,\dots, E_m; F)$. Suppose, by contradiction, that there exist $(a_1,\dots, a_m) \in E_1 \times \cdots \times E_m$ and sequences $(x_j^{(i)})_{j=1}^\infty$ in $E_i$ such that
	\[
	x_j^{(i)} \overset{w}{\longrightarrow} a_i \quad \text{for all } i=1,\dots,m,
	\]
	but
	\[
	T(x_j^{(1)},\dots, x_j^{(m)}) \nrightarrow T(a_1,\dots, a_m).
	\]
	Then, passing to a subsequence if necessary, there exists $\varepsilon > 0$ such that
	\[
	\|T(x_j^{(1)},\dots, x_j^{(m)}) - T(a_1,\dots, a_m)\| \ge \varepsilon
	\]
	for all $j \in \mathbb{N}$.
	
	By the Hahn-Banach Theorem, for each $j$ there exists $f_j \in F'$ with $\|f_j\| = 1$ such that
	\[
	f_j\big(T(x_j^{(1)},\dots, x_j^{(m)}) - T(a_1,\dots, a_m)\big)
	= \|T(x_j^{(1)},\dots, x_j^{(m)}) - T(a_1,\dots, a_m)\|.
	\]
	
	Since $(f_j)$ lies in the unit ball of $F'$, by the Banach--Alaoglu Theorem, Eberlein–Šmulian Theorem and the reflexivity of $F$ there exists a subsequence $(f_{j_k})$ and $f \in F'$ such that $f_{j_k} \overset{w}{\longrightarrow} f$. So,
	%\[
	%f_{j_k} \overset{w^*}{\longrightarrow} f.
	%\]
	%As $F$ is reflexive, the weak* and weak topologies on $F'$ coincide, hence
	%\[
	%f_{j_k} \overset{w}{\longrightarrow} f.
	%\]
	
	%Therefore,
	\begin{align*}
		\varepsilon 
		&\le \|T(x_{j_k}^{(1)},\dots, x_{j_k}^{(m)}) - T(a_1,\dots, a_m)\| \\
		&= f_{j_k}\big(T(x_{j_k}^{(1)},\dots, x_{j_k}^{(m)}) - T(a_1,\dots, a_m)\big) \\
		&= f_{j_k}\big(T(x_{j_k}^{(1)},\dots, x_{j_k}^{(m)})\big) - f_{j_k}\big(T(a_1,\dots, a_m)\big).
	\end{align*}
	
	Thus,
	\begin{align*}
		\varepsilon 
		&\le \left| f_{j_k}\big(T(x_{j_k}^{(1)},\dots, x_{j_k}^{(m)})\big) - f\big(T(a_1,\dots, a_m)\big) \right| \\
		&\quad + \left| (f - f_{j_k})\big(T(a_1,\dots, a_m)\big) \right|.
	\end{align*}
	
	Since, $T \in \mathcal{L_{WDP}^{\text{ev}}}(E_1,\dots, E_m; F)$ and $f_{j_k} \overset{w}{\longrightarrow} f$, the first term converges to $0$, and the second term also converges to $0$. This is a contradiction.
	
	Therefore,
	\[
	\mathcal{L_{WDP}^{\text{ev}}}(E_1,\dots, E_m; F) \subset\mathcal{L}_{wsc}(E_1,\dots, E_m; F),
	\]
	and the proof is complete.
\end{proof}

It is important to note that, in Theorem \ref{WDP=WDP}, the only special condition imposed was on the Banach space $F$; the other Banach spaces remain arbitrary. Assuming that all Banach spaces $E_1,\dots, E_m$ possess the Schur property, and this time that $F$ is arbitrary, it is an immediate consequence of Proposition \ref{DPEVS} that
\begin{equation*}
\mathcal{L_{WDP}^{\text{ev}}}(E_1,\dots, E_m; F) = \mathcal{L}(E_1,\dots, E_m; F).
\end{equation*}

%Another immediate consequence of Theorem \ref{DPEVS} is the following result.

%\begin{proposition}
%	Let $E_1,\dots, E_m$ and $F$ be Banach spaces. If $E_i$ has the Schur property for all $i=1,\dots, m$, then every continuous multilinear mapping from $E_1\times\cdots\times E_m$ into $F$ is weakly Dunford--Pettis at every point. In particular,
%	\[
%	\mathcal{L_{WDP}^{\text{ev}}}(E_1,\dots, E_m; F) = \mathcal{L}(E_1,\dots, E_m; F).
%	\]
%\end{proposition}

%\begin{proof}
%	By Proposition \ref{DPEvsubWDPEv}, we have
%	\[
%	\mathcal{L_{DP}^{\text{ev}}}(E_1,\dots, E_m; F)
%	\subset 
%	\mathcal{L_{WDP}^{\text{ev}}}(E_1,\dots, E_m; F)
%	\subset 
%	\mathcal{L}(E_1,\dots, E_m; F).
%	\]
%	Since each $E_i$ has the Schur property, Theorem \ref{DPEVS} ensures that
%	\[
%	\mathcal{L_{DP}^{\text{ev}}}(E_1,\dots, E_m; F)
%	=
%	\mathcal{L}(E_1,\dots, E_m; F).
%	\]
%	The conclusion follows immediately.
%\end{proof}

\begin{proposition}
	Let $E_1,\dots, E_m$ and $F$ be Banach spaces. If the projective tensor product $E_1 \hat{\otimes}_{\pi} \cdots \hat{\otimes}_{\pi} E_m$ has the Schur property, then
	\[
	\mathcal{CC} \circ \mathcal{L}(E_1,\dots, E_m; F)
	=
	\mathcal{L_{WDP}^{\text{ev}}}(E_1,\dots, E_m; F).
	\]
\end{proposition}

%\begin{proof}
%	Since $E_i$ has the Schur property for all $i=1,\dots, m$ and $E_1 \hat{\otimes}_{\pi} \cdots \hat{\otimes}_{\pi} E_m$ has also the Schur property, then
%	\begin{equation*}
%		\mathcal{L_{WDP}^{\text{ev}}}(E_1,\dots, E_m; F) = \mathcal{L}(E_1,\dots, E_m; F) = \mathcal{DP} \circ %\mathcal{L}(E_1,\dots, E_m; F).
%	\end{equation*}
%\end{proof}

%\begin{proof}
%	By the assumptions, it follows from \cite[Section 4]{BT15} that
%	\[
%	\mathcal{DP} \circ \mathcal{L}(E_1,\dots, E_m; F)
%	=
%	\mathcal{L}(E_1,\dots, E_m; F).
%	\]
%	On the other hand, since each $E_i$ has the Schur property, we know that
%	\[
%	\mathcal{L_{WDP}^{\text{ev}}}(E_1,\dots, E_m; F)
%	=
%	\mathcal{L}(E_1,\dots, E_m; F).
%	\]
%	The result follows immediately.
%\end{proof}
%%%%%%%%%%%%%%%%%%%%%%%%%%%%%%%

To ensure
\begin{equation*}
\mathcal{L_{WDP}^{\text{ev}}}(E_1,\dots, E_m; F) = \mathcal{L}(E_1,\dots, E_m; F),
\end{equation*}
we needed to impose a strong condition on the Banach spaces $E_i$: they must have the Schur property. The next result shows that if we assume some $E_i$ has the Dunford-Pettis property, we can still obtain the same identification.

\begin{theorem}
	Let $E_1,\dots, E_m$ and $F$ be Banach spaces. Fix $i_{0} \in \{1,2,\dots, m\}$ so that $E_{i_0}$ has the Dunford-Pettis property, and suppose that $E_i$ with $i \in \{1,2,\dots, m\}\setminus{\{i_0\}}$ has the Schur property. Then every continuous multilinear mapping from $E_1\times\cdots\times E_m$ into $F$ is weakly Dunford--Pettis at every point. In particular,
	\begin{equation*}
	\mathcal{L_{WDP}^{\text{ev}}}(E_1,\dots, E_m; F) = \mathcal{L}(E_1,\dots, E_m; F).
	\end{equation*}
\end{theorem}

\begin{proof}
	We will only consider the case $i_0=1$; the other cases follow analogously. Let $E_1,\dots,E_m$ and $F$ be Banach spaces, and assume that each $E_i$ has the Schur property for $i=2,\dots,m$ and $E_1$ has the Dunford-Pettis property. Let $T \in \mathcal{L}(E_1,\dots,E_m;F)$. We show that $T \in \mathcal{L_{WDP}^{\mathrm{ev}}}(E_1,\dots,E_m;F)$. let $(x_j^{(i)})_{j=1}^{\infty} \subset E_i$ and $(f_j)_{j=1}^{\infty} \subset F^{'}$ such that $x_j^{(i)} \xrightarrow{w} x_i \in E_i$, $i=1,\dots, m$, and $f_j \xrightarrow{w} f \in F^{'}$. With a simple computation, we can see that 

\begin{align*}
&\left|f_j\left(T(x_j^{(1)},\ldots,x_j^{(m)})\right)
-f\left(T(x_1,\ldots,x_{m})\right)\right|
\\
&\le{}
\left|f_j\left(
T(x_j^{(1)}-x^{(1)},x_{2},\ldots,x_{m})
\right)\right|
\\
&+
\sum_{k=2}^{m}
\left|f_j\left(
T\left(
x_j^{(1)},\ldots,x_j^{(k-1)},
x_j^{(k)}-x_{k},
x_{k+1},\ldots,x_{m}
\right)
\right)\right|
\\
&+
|(f_j-f)\left(T(x_{1},\ldots,x_{m})\right)|.
\end{align*}
Since $E_i$ has the Schur property for $i=2,\dots,m$ and $E_1$ has the Dunford-Pettis property, and since $f_j \xrightarrow{w} f$ in $F^{'}$ implies $f_j\circ t \xrightarrow{w} f\circ t$ in $E^{'}$, for every $t \in \mathcal{L}(E; F)$, then the desired result follows immediately from the inequality above.
\end{proof}

\begin{example}
Considering $E_1 = \ell_{\infty}$ and $E_2 = \cdots =E_m = \ell_1$, for every Banach space $F$ we have
\begin{equation*}
		\mathcal{L_{WDP}^{\text{ev}}}(\ell_{\infty}, \ell_1,\dots,\ell_1;F)
		= \mathcal{L}(\ell_{\infty}, \ell_1,\dots,\ell_1;F)
	\end{equation*}
because, $\ell_1$ has the Schur property and $\ell_{\infty}$ has the Dunford-Pettis property.
\end{example}

\subsection*{Declaration}

{\bf Conflict of interest} The authors declare that they have no conflict of interest.

\end{document}